\documentclass[12pt]{article}
  
\usepackage[english]{babel}

\usepackage{amsmath}
\usepackage{amsthm}
\usepackage{amssymb}
\usepackage{amsfonts}
\usepackage{amsbsy}

\columnwidth\textwidth

    \newcommand{\cb}{{\mathcal B}}

\newcommand{\Ee}{\mathbb{E}}

\newcommand{\vf}{{\varphi}}

\newcommand{\ds}{\displaystyle}

\newtheorem{thm}{Theorem}[section]
\newtheorem{lem}[thm]{Lemma}
\newtheorem{prop}[thm]{Proposition}

\newtheorem{exam}[thm]{Example}

\newtheorem*{remark}{Remark} 

\numberwithin{equation}{section}

\begin{document}

\noindent
{\Large\bf Measure-valued branching processes associated with\\ 
Neumann nonlinear semiflows}\\[2mm]

\noindent 
{\bf  Viorel Barbu}\footnote{Al. I. Cuza University  and Octav Mayer Institute of Mathematics (Romanian Academy), 
Blvd. Carol I, No. 8, 700506 Ia\c si, Romania. E-mail: vbarbu41@gmail.com}, 
{\bf Lucian Beznea}\footnote{Simion Stoilow Institute of Mathematics  of the Romanian Academy, 
P.O. Box \mbox{1-764,} RO-014700 Bucharest, Romania, University of Bucharest, Faculty of Mathematics and Computer Science, 
and Centre Francophone en Math\'ematique de Bucarest. 
E-mail: lucian.beznea@imar.ro}


\vspace{7mm}

\noindent {\small{\bf Abstract.}  
We construct a measure-valued branching Markov process associated with a nonlinear boundary value problem, 
where the boundary condition has a nonlinear pseudo monotone branching mechanism term $-\beta$, 
which includes as a limit  case $\beta(u) = - u^{m}$, with $0 < m < 1$. 
The process is then used in the probabilistic representation of the solution of  the parabolic problem
associated with a nonlinear Neumann boundary value problem.
In this way the classical association of the superprocesses to the Dirichlet boundary value problems 
also holds for the nonlinear Neumann boundary value problems.
It turns out that the obtained branching process behaves 
on the measures carried by the given open set like the linear continuous semiflow, 
induced by the reflected Brownian motion, while the branching occurs 
on the measures having non-zero traces on the boundary of the open set, 
with the behavior  of the $(-\beta)$-superprocess, having as spatial motion the process 
on the boundary associated to the reflected Brownian motion.
\\[2mm]

}\\


\section{Introduction}

\vspace{0.3cm}

Let $\mathcal{O}$ be a bounded, open subset of $\mathbb{R}^d$, $d \geqslant 1$, 
with smooth boundary $\Gamma$ (for instance, of class $C^2$). 
Consider the nonlinear parabolic problem
$$
\left\{ \begin{array}{ll}
\dfrac{\partial u}{\partial t} - \frac{1}{2} \Delta u + \alpha u = 0 & {\rm in} \;  (0, \infty) \times \mathcal{O},\\[4mm]
\dfrac{\partial u}{\partial \nu} + \beta(u) = g & {\rm on} \; \Gamma,\\[4mm]
u(0, \cdot) = f & {\rm in} \; \mathcal{O},
\end{array}\right. \leqno(1.1)
$$
where $\frac{\partial}{\partial \nu}$ is the outward normal derivative to the boundary 
$\Gamma$ of  $\mathcal{O}$, 
$g$ is a positive continuously differentiable function on $\Gamma$, 
$f \in C(\overline{\mathcal{O}})$, $\alpha \in \mathbb{R}^{\ast}_+$, 
and $\beta : \mathbb{R} \longrightarrow \mathbb{R}_-$ is the following continuous mapping
$$
\beta(u) = \left\{ \begin{array}{ll}
\ds\int_{0}^{\infty}(e^{-su} - 1)\eta(ds) - bu, & {\rm if} \; u \geqslant 0,\\[5mm]
0, & {\rm if} \; u < 0,
\end{array}\right. \leqno(1.2)
$$
with $\eta$ a positive measure on $\mathbb{R}_+$ such that 
$\int_{\mathbb{R}_+}\!\! s\, \eta(ds) < \infty$ and $b \in \mathbb{R}_+$.

We assume that
$$
\mathop{\int}\limits_{\mathbb{R}_+}\!\!  s\, \eta(ds) + 
b \leqslant \gamma 
:= \mathop{\inf}\limits_{v \in H^1(\mathcal{O})} \dfrac{\frac{1}{2} \| \nabla v \|^2_{L^2(\mathcal{O})} 
+ \alpha\| v \|^2_{L^2(\mathcal{O})}}{\| v \|^2_{L^2(\Gamma)}}. \leqno(1.3)
$$
Note that the last inequality is equivalent with 
the property of the function $u\longmapsto \beta(u)+\gamma u$ to be nondecreasing.
Moreover in this case $\beta$ is a Lipschitz function.

The function $-\beta$ is called {\it branching mechanism}.
An example of $\beta$ satisfying $(1.3)$ is 
$$
\beta_N(u)= \dfrac{m}{\Gamma(1 - m)}\mathop{\int}\limits_{0}^{\ \ \ N} \dfrac{e^{-su}-1}{s^{m + 1}} ds.
$$
with $N>0$ and $0< m < 1$.
A limit case is therefore $ au^{m}$ for  a convenient number $a < 0$, since
$$
- u^{m} = \dfrac{m}{\Gamma(1 - m)}\mathop{\int}\limits_{0}^{\infty} \dfrac{e^{-su}-1}{s^{m + 1}} ds
= \lim_{N\to \infty} \beta_N(u).
$$

If $\beta = \alpha = 0$ and $g = 0$ then the solution of the linear problem $(1.1)$ 
is given by the transition function of the reflected Brownian motion 
$B = (B_t)_{t \geqslant 0}$ on $\mathcal{O}$: $u(t, \cdot) = \mathbb{E}^{\cdot}(f(B_t))$, $t \geqslant 0$, 
where $f$ is a bounded, real-valued Borel measurable function on $\mathcal{O}$.

The first aim of this paper is to show that the solution of $(1.1)$ 
admits a probabilistic interpretation if $\beta$ is given by  $(1.2)$, 
which is similar with what happens in the linear case $(\beta = 0)$.
More precisely, there exists a branching Markov process $X = (X_t)_{t \geqslant 0}$ with state space 
the set $M(\overline{\mathcal{O}})$ of all positive finite measures on $\overline{\mathcal{O}}$, such that the solution of $(1.1)$ is
$$
u(t, x) = -\ln \mathbb{E}^{\delta_x}(e_f(X_t)), \; t \geqslant 0, \; x \in \overline{\mathcal{O}}, \leqno(1.4)
$$ 
where for  a Borel, positive, real-valued function $f$ on $\overline{\mathcal{O}}$ 
we considered the {\it exponential mapping \;} $e_f : M(\overline{\mathcal{O}}) \longrightarrow [0, 1]$, defined as
$$
e_f(\mu) := e^{-\int f d\mu} \quad {\rm for \; all} \; \mu \in M(\overline{\mathcal{O}}).
$$ 

The first step of our approach is to prove the existence of the solution of $(1.1)$. 
We consider the maximal monotone operator $\mathcal{A}$ associated to  $(1.1)$ 
and we show that it is the infinitesimal generator of a nonlinear semigroup of contractions 
$(V_t)_{t \geqslant 0}$ on $L^2({\mathcal{O}})$, such that 
$u(t, \cdot) := V_tf$ is a solution of  $(1.1)$ 
for each $f$ from the domain of $\mathcal{A}$. 
If $ 1 \leqslant d \leqslant 3$ then $(V_t)_{t \geqslant 0}$ induces a $C_0$-semigroup of (nonlinear) contractions on $C(\overline{\mathcal{O}})$.

The second step is to prove that the map $f \longmapsto V_tf(x)$, $x \in \overline{\mathcal{O}}$, 
is negative definite on $C_+(\overline{\mathcal{O}})$ 
(:= the set of all positive continuous functions on $\overline{\mathcal{O}}$). 
We use essentially an approximating process in solving  $(1.1)$ and a negative definiteness  property of the mapping $-\beta$.

The last step is to follow the so called semigroup approach in order to construct the claimed measure-valued branching process; 
see \cite{Wat68}, \cite{Fi88}, \cite{Dy02},  \cite{LeGa99}, \cite{Li11}, and \cite{Be11}.

We can describe the infinitesimal generator of the branching process $X$, 
which shows that it behaves as the linear semiflow $t \longmapsto \mu \circ P_t$, $t \geqslant 0$, $\mu \in M(\mathcal{O})$ 
(:= the set of all positive finite measures on $\mathcal{O}$), 
where $(P_t)_{t \geqslant 0}$ is the transition function of the reflected Brownian motion. 
The branching property of $X$ holds on the measures with non-zero traces on the boundary $\Gamma$ of $\mathcal{O}$;
for more details see the final remark of this paper.

Formula $(1.4)$  suggests that we can compare the measure-valued process $X$ with the 
$(B, \beta_0)$-superprocess (in the sense of Dynkin; see \cite{Dy02}), where
$$
\beta_0(x, u) = \left\{ \begin{array}{ll}
- \beta(u) + g(x), & {\rm if} \; x \in \Gamma, \; u \geqslant 0,\\[3mm]
0, & {\rm if} \; x \in \mathcal{O}, \; u < 0.
\end{array}\right.
$$

We may conclude that the classical association of the superprocesses  to the Dirichlet boundary value problems also holds 
for the nonlinear Neumann boundary value problems. 
However, due to the boundary flux induced by $g$, 
this branching process is no longer conservative if $g \not\equiv 0$.

The second aim of this paper is to prove that the solution of the nonlinear parabolic problem with the dynamic flux on the boundary
$$
\left\{ \begin{array}{ll}
 \frac{1}{2} \Delta u - \alpha u = 0 & {\rm in} \; \mathcal{O},\\[3mm]
\dfrac{\partial v}{\partial t} + \dfrac{\partial u}{\partial \nu} + \beta(v) = 0 & {\rm on} \; \Gamma,\\[3mm]
v = u|_{\Gamma}, & u \in C(\overline{\mathcal{O}}),
\end{array}\right.
$$
also admits a probabilistic interpretation, which is related to the branching process associate with the equation $(1.1)$ with $g\equiv 0$. 
It turns out that the associated measure-valued branching process is precisely 
the $(Z, \alpha- \beta)$-superprocess on $M(\Gamma)$
(:= the set of all positive finite measures on $\Gamma$), 
where $Z = (Z_t)_{t \geqslant 0}$ is the boundary process  (on $\Gamma$), 
induced by the reflected Brownian motion, or equivalently, by the classical (linear) Neumann boundary value problem.
Comparing the infinitesimal operators, one can see that the measure-valued branching process $X$ behaves on $M(\Gamma)$ as this
superprocess; cf. Final Remark.
So, for the process $X$ on $M(\overline{\mathcal{O}})$, the role of the "boundary process" is played by the
$(Z, \alpha- \beta)$-superprocess  on $M(\Gamma)$.

Finally, we thank  the anonymous referee for carefully reading the manuscript
and for his valuable comments.

\section{\mbox{Nonlinear $C_0$-semigroups generated by the Neumann} problem}  
Everywhere in the following we assume that condition $(1.3)$ holds.

Let $H = L^2(\mathcal{O})$ denote the space of al real-valued square integrable functions on $\mathcal{O}$ with the scalar product
$$
\langle u, v \rangle_2 = \int_{\mathcal{O}} u \, v dx \, \mbox{ for all } u, v \in   L^2(\mathcal{O}), 
$$
and the norm $|u|_2 :=  \langle u, u \rangle_2^{\frac{1}{2}}$. By $H^k(\mathcal{O})$, $k = 1,2$, 
we denote the standard Sobolev space in $L^2(\mathcal{O})$ and $C(\overline{\mathcal{O}})$ 
denotes the space of all continuous functions on $\overline{\mathcal{O}}$, endowed with the supremum norm $||\cdot||_{C(\overline{\mathcal{O}})}$. 
Let further $\sigma$ denote  the surface measure on $\Gamma$.

Define the nonlinear operator $\mathcal{A} : D(\mathcal{A}) \subseteq  H \longrightarrow H$,
$$
\left\{ \begin{array}{ll}
D(\mathcal{A}) := \{ u \in H^2(\mathcal{O}) : \dfrac{\partial u}{\partial \nu} + \beta(u) = g  \mbox{  on} \; \Gamma\},\\[3mm]
\mathcal{A} u := - \frac{1}{2} \Delta u+ \alpha u \quad   \mbox{for all} \; u \in D(\mathcal{A}).
\end{array}\right. \leqno(2.1)
$$

Note that $D(\mathcal{A})$ is dense in $H$.
We have for all $u, v \in D(\mathcal{A})$
$$
\langle \mathcal{A} u - \mathcal{A} v, u - v \rangle_2 = 
\int_{\mathcal{O}}(\frac{1}{2}  |\nabla(u - v)|^2 + \alpha(u - v)^2) dx + \frac{1}{2} \int_{\Gamma}(\beta(u) - \beta(v))(u - v) d\sigma 
$$
$$
\geqslant 
\int_{\mathcal{O}}(\frac{1}{2} |\nabla(u - v)|^2 + \alpha(u - v)^2) dx- \frac{\gamma}{2}  \int_{\Gamma}|u - v|^2 d\sigma \geqslant 0,
$$
because
$$
\int_{\mathcal{O}}  (\frac{1}{2} |\nabla v|^2 + \alpha v^2) dx \geqslant \gamma \int_\Gamma v^2 d\sigma \quad \mbox{ for all } v\in H^1,  \leqno{(2.2)}
$$
where the first inequality holds by the monotonicity of the map $r \longmapsto \beta(r) + \gamma  r$ 
and the second one is a consequence of condition $(1.3)$. 
Hence $\mathcal{A}$ is monotone in $H \times H$.

It should be said also that $\mathcal{A}$ is a {\it potential operator}, $\mathcal{A} = \partial \Phi$, 
where  $\Phi : L^2(\mathcal{O}) \longrightarrow (-\infty, +\infty]$ is the lower semicontinuous, convex function   defined as
$$
\Phi(u) := \left\{ \begin{array}{ll}
\dfrac{1}{4} \ds\int_{\mathcal{O}} |\nabla u|^2 dx + \frac{1}{2}  \int_{\Gamma} j(u) d\sigma + \dfrac{\alpha}{2} \int_{\mathcal{O}} u^2 dx & {\rm if} \; u \in H^1(\mathcal{O}),\\[6mm]
+ \infty & {\rm elsewhere},
\end{array}\right.
$$
with $j(r) := \int_{0}^{r} \beta(s) ds$, $r \in \mathbb{R}$. 
Moreover, by a fundamental result due to H. Br\'ezis \cite{Brez72}, 
it follows that $R(I + \lambda \mathcal{A}) = H$ for all $\lambda > 0$ (see also \cite{Bar10}, p. 99). 
In fact,  the Br\'ezis regularity result is given for the elliptic problem 
$u - \lambda \Delta u = f$ in $\mathcal{O}$, $\dfrac{\partial u}{\partial v} + \widetilde{\beta}(u) = 0$ on $\Gamma$, where $\widetilde{\beta}$ is monotone.
The extension to the present case (that is $r \longmapsto \beta(r) + \gamma r$ is monotone) follows however by the same argument.

By standard existence theory of the Cauchy problem in Banach spaces (see e.g. \cite{Bar10}, p. 163) 
it follows that $\mathcal{A}$ is the infinitesimal generator of a continuous semigroup of contractions $(V_t)_{t \geqslant 0}$ in $H$,
that is 
$$
|V_t f - V_t \bar f|_2 \leqslant |f - \bar f|_2  \; \mbox{ for all } \; t \geqslant 0, f, \bar f \in H.
$$
Moreover, it turns out that $u(t) := V_tf$, $t \geqslant 0$, is the unique solution of $(1.1)$ if $f \in D(\mathcal{A})$ and 
$$
\left\{ \begin{array}{l}
\dfrac{d^+}{dt} V_tf + \mathcal{A}V_tf = 0 \quad {\rm for \; all} \; t \geqslant 0,\\[3mm]
V_0f = f
\end{array}\right. \leqno(2.3)
$$

\noindent (cf. Theorem 4.12 from \cite{Bar10}).

The next two propositions collect other properties of $(V_t)_{t \geqslant 0}$.

\vspace{0.3cm}

\begin{prop} 
Let $f \in H$ and define the mapping $u : [0, \infty) \longrightarrow H$ as $u(t) := V_tf$, $t \geqslant 0$. 
Then the following assertions hold

$(i)$ If $f \in D(\mathcal{A})$ and $T > 0$ then $u(t) \in D(\mathcal{A})$  for all  $t \geqslant 0$ and 
$$
u \in  C([0, T]; H) \cap L^{\infty}(0, T; H^2(\mathcal{O})), \; \dfrac{\partial u}{\partial t} \in L^{\infty}(0, T; H^2(\mathcal{O})), 
$$
$$
|\mathcal{A}u(t)|_2 \leqslant |\mathcal{A}f|_2 \quad for \; all \; t > 0. 
$$

$(ii)$ The function $t \longmapsto u(t)$ is right continuous from $(0, \infty)$ to $H^2(\mathcal{O})$ and so, if $1 \leqslant d \leqslant 3$, 
then since $H^2(\mathcal{O}) \hookrightarrow C(\overline{\mathcal{O}})$, we have
$$
\mathop{\lim}\limits_{s \searrow t} \| u(s) - u(t) \|_{C(\overline{\mathcal{O}})} = 0 \quad for \; all \; t \geqslant 0.
$$

$(iii)$  If $f \in H$ then $u \in C([0, \infty); H)$ and $\dfrac{du}{dt}$, 
$\mathcal{A}u \in L^\infty(\delta, T;  H)$  for all  $\delta \in (0, T)$. 
If $c \in \mathbb{R_+}$ and $c \leqslant f$ then  $e^{-\alpha t}c \leqslant V_t f$. 
In particular, if $f \geqslant 0$ then $V_tf \geqslant 0$ for all $t \geqslant 0$.
Moreover, if  $f_1, f_2 \in H$, $f_1 \leqslant f_2$ a.e. in $\mathcal{O}$, 
then $V_t f_1 \leqslant V_t f_2$  a.e. in $\mathcal{O}$.
\end{prop}

\begin{proof}
Assertions $(i)$  and $(ii)$ are consequences of the existence theory of the Cauchy problem associated 
with nonlinear maximal monotone operators in Hilbert spaces (see e.g. \cite{Bar10}). 
To prove assertion $(iii)$,  assume first that $c \leqslant f$ and rewrite equation $(1.1)$ as
$$
\left\{ \begin{array}{ll}
\dfrac{\partial}{\partial t}(e^{\alpha t}u - c) - \frac{1}{2}  \Delta(e^{\alpha t}u - c) = 0 & {\rm in} \; (0, T) \times \mathcal{O},\\[3mm]
\dfrac{\partial}{\partial \nu}(e^{\alpha t}u - c) + e^{\alpha t}(\beta(u) - g) = 0 & {\rm on} \; (0, T) \times \Gamma,
\end{array}\right.
$$
multiply by $(e^{\alpha t}u - c)^-$ and integrate on $(0, t) \times \mathcal{O}$, for all $t \in (0, T)$ 
to get after some calculation that $(e^{\alpha t}u - c)^- \leqslant 0$ a.e. on $(0, T) \times \mathcal{O}$. 

One proves similarly  that $(V_tf_2-V_tf_1)^{+}=0$ on $(0,T)\times \mathcal{O}$, 
involving as above  $(2.2)$ and the monotonicity of the function $u\longmapsto \beta(u) +\gamma u$.
Indeed, if $u_i := V_t f_i$, $i=1,2$,  and $v=u_1-u_2$, then 
$$
\left\{ \begin{array}{ll}
\dfrac{\partial}{\partial t}v -  \frac{1}{2}  \Delta v +\alpha v  = 0 & {\rm in} \; (0, T) \times \mathcal{O},\\[3mm]
\dfrac{\partial}{\partial \nu}v + \beta(u_1) - \beta(u_2) = 0 & {\rm on} \;  \Gamma.
\end{array}\right.
$$
Multiplying by  $v^+$ and integrating we get for all $t\geqslant 0$
$$
\| v^+(t) \|^2_{L^2} +\frac{1}{2} \int_0^t \int_{\mathcal{O}}|\nabla v^+|^2 ds dx 
+ \alpha\!\!  \int_0^t \int_{\mathcal{O}}|v^+|^2 ds dx 
+ \frac{1}{2}  \int_0^t \int_{\Gamma} (\beta(u_1) - \beta(u_2)) v^+ ds d\sigma =0.
$$
Since $\beta(u_2) - \beta(u_1)\leqslant  \gamma(u_1-u_2)$ on the set $[v^+ >0]$,  we obtain
$$
\| v^+(t) \|^2_{L^2} +\frac{1}{2} \int_0^t \int_{\mathcal{O}}|\nabla v^+|^2 ds dx +
\alpha \!\! \int_0^t \int_{\mathcal{O}}|v^+|^2 ds dx 
\leqslant \frac{\gamma}{2}  \int_0^t \int_{\Gamma} |v^+|^2  ds d\sigma.
$$
By $(2.2)$ we conclude that $ \| v^+(t) \|^2_{L^2}=0$ for all $t\geqslant 0$, as claimed.
\end{proof}

An interesting feature of the family $(V_t)_{t\geqslant 0}$ is that it is a $C_0$-semigroup 
on $C(\overline{\mathcal{O}})$, 
namely,  we have the following result.

\begin{prop}     
Assume that $1 \leqslant d \leqslant 3$. 
Then the family $(V_t)_{t \geqslant 0}$ induces a $C_0$-semigroup of 
$\alpha$-quasi-contractions on $C(\overline{\mathcal{O}})$,
i.e., for each $f, \bar{f}   \in C(\overline{\mathcal{O}})$ we have
$$
V_tf \in C(\overline{\mathcal{O}}) \; {\rm and} \; \| V_tf - V_t\bar{f} \|_{*} 
\leqslant e^{\alpha t} \| f - 
\bar{f} \|_{*} \ \ for \; all \; t \geqslant 0,
$$
$$
V_tV_sf = V_{t + s}f \; for \; all \; t, s \geqslant 0,
$$
$$
\mathop{\lim}\limits_{t \searrow 0}\| V_tf - f \|_{*} = 0.
$$
Here  $\| \cdot  \|_{*}$ is a norm on $C(\overline{\mathcal{O}})$ 
which is equivalent with the supremum norm,
to be defined below.
In addition we have $V_t(S) \subseteq S$ for all $t \geqslant 0$, where $S := C_+(\overline{\mathcal{O}})$.
\end{prop}

\begin{proof} 
Since $1 \leqslant d \leqslant 3$ we have $H^2(\mathcal{O}) \subseteq C(\overline{\mathcal{O}})$ and so, 
by Proposition 2.1 $V_tf \in C(\overline{\mathcal{O}})$ for all $f \in H$ and $t \geqslant 0$. 

We show now that  the operator $\mathcal{A}_0$ 
defined as 
$$
\mathcal{A}_0 u := - \frac{1}{2}\Delta u + \alpha u\  \mbox{ for } 
u \in D(\mathcal{A}_0) := \{ u \in H^2(\mathcal{O}) : \Delta u \in C(\overline{\mathcal{O}}), \dfrac{\partial u}{\partial \nu} + \beta(u) 
= g \; {\rm on} \; \Gamma \}
$$ 
is quasi-$m$-accretive in $C(\overline{\mathcal{O}})$, i.e.,  for some $\lambda_0 > 0$ and  any $0< \lambda < \lambda_0$ we have
$$
\| ((1+\lambda\alpha)I + \lambda \mathcal{A}_0)^{-1}v - ((1+\lambda\alpha)I + \lambda \mathcal{A}_0)^{-1}\bar{v} \|_{*} 
\leqslant \| v - 
\bar{v} \|_{*} \; {\rm for \; all} \; v, \bar{v} \in C(\overline{\mathcal{O}}).
$$
Indeed, if $u := ((1+\lambda\alpha)I + \lambda \mathcal{A}_0)^{-1}v$ 
and  
$\bar{u} := ((1+\lambda\alpha)I + \lambda \mathcal{A}_0)^{-1}\bar{v}$, 
then
$$
(1+\lambda\alpha)(u - \bar{u}) - \frac{\lambda}{2}  \Delta(u - \bar{u}) + \lambda \alpha(u - \bar{u}) = v - \bar{v} \; {\rm in} \; \mathcal{O} \; {\rm and} 
\; \dfrac{\partial u}{\partial \nu}(u - \bar{u}) + \beta(u) - \beta(\bar{u}) = 0 \; {\rm on} \; \Gamma.
$$
We choose now $\vf\in C^2(\overline{\mathcal{O}})$ such that
$$
 \vf \geqslant \rho>0\  \mbox{ and }  \
 2\alpha - \vf \Delta(\frac{1}{\vf}) \geqslant 0 \mbox{ in }  \overline{\mathcal{O}}, 
\  \ \frac{1}{\vf}\, \frac{\partial\vf}{\partial\nu}\geqslant L  \mbox{ on } \Gamma,\mbox{ with }  L:= \mbox{\rm Lip}(\beta).  \leqno{(2.4)}
$$
An example of such a function $\vf$ is
$\vf(x):= exp(\delta\psi(x))$, 
where $\delta> 0$ is sufficiently small and $\Delta\psi=K$ in $\mathcal{O}$, 
$ \frac{\partial\psi}{\partial\nu}=\frac{L}{\delta}$ on $\Gamma$; 
$K,$ $\delta$ are such that $Km(\mathcal{O})=\frac{L}{\delta}\sigma(\Gamma)$,
where $m$ is the Lebesgue measure.

We set 
$$
\| z  \|_{*}:= \sup \{ | z(x)\vf({x}) | : x\in \overline{\mathcal{O}}\}
\mbox{ for all } z\in C(\overline{\mathcal{O}})\ \mbox{ and } 
c := \| v - \bar{v} \|_{*}.
$$ 
Let $y:= u \vf$ and $\bar{y}:= \bar{u}\vf$.
We have
$$ 
y- \bar{y}  - \frac{\lambda}{2} \Delta(y - \bar{y}) + \lambda \alpha(y - \bar{y})
\frac{\lambda}{2} \vf \Delta (\frac{1}{\vf})(y- \bar{y} )-
 \frac{\lambda}{2} \Delta(y - \bar{y})=\vf (v- \bar{v} )\  \mbox{ in }  \mathcal{O},
$$
$$
\frac{\partial}{\partial \nu}(y - \bar{y}) - \frac{1}{\vf} \frac{\partial\vf}{\partial \nu} (y - \bar{y})
+ \beta(\frac{1}{\vf} y) - \beta(\frac{1}{\vf} \bar{y}) = 0 \; {\rm on} \; \Gamma. 
$$
This yields
$$
(y-\bar{y}- c) -  \frac{\lambda}{2} \Delta(y - \bar{y}- c)- 
 \lambda \vf \nabla (\frac{1}{\vf}) \nabla (y - \bar{y}- c)
 +(\lambda \alpha -\frac{\lambda}{2}\vf \Delta (\frac{1}{\vf}) (y - \bar{y}-c)
$$
$$
= \vf(v-\bar{v}) - (\lambda \alpha -\frac{\lambda}{2} \vf \Delta (\frac{1}{\vf}) )c -c \leqslant 0\ \mbox{ in } \mathcal{O},
$$
$$
\frac{\partial}{\partial\nu} (y-\bar{y}- c)= 
\frac{1}{\vf} \frac{\partial\vf}{\partial \nu} (y - \bar{y}) -\beta (\frac{1}{\vf}) - \beta (\frac{1}{\vf} \bar{y})\ \mbox{ on }  \Gamma.
$$
Multiplying by $(y-\bar{y}- c)^+$  and integrating  on $\mathcal{O}$, we get via Green's formula
$$
\int_\mathcal{O} ((y-\bar{y}- c)^+)^2 dx + \frac{\lambda}{2} \int_\mathcal{O} |\nabla (y-\bar{y}- c)^+ |^2 dx
$$
$$
\leqslant \frac 1 2 \int_\mathcal{O}  |(y-\bar{y}- c)^+|^2 dx 
+\frac{\lambda^2}{2} \int_\mathcal{O} |\vf \nabla(\frac{1}{\vf})|^2  ((y-\bar{y}- c)^+)^2 dx 
$$
$$
+ \int_\Gamma (\beta(\frac{1}{\vf}y) - \beta(\frac{1}{\vf} \bar{y})) (y-\bar{y} - c)^+ d \sigma
- \int_\Gamma \frac{1}{\vf} \frac{\partial\vf}{\partial \nu} (y- \bar{y}) (y-\bar{y} - c)^+ d \sigma .
$$
Taking into account that by $(2.4)$
$$
(\beta(\frac{1}{\vf}y) - \beta(\frac{1}{\vf} \bar{y})) (y-\bar{y} - c)^+ 
- \frac{1}{\vf} \frac{\partial\vf}{\partial \nu} (y- \bar{y}) (y-\bar{y} - c)^+  \leqslant 0\ \mbox{ on } \Gamma ,
$$
it follows that for $\lambda\in (0, \lambda_0)$, $\lambda_0$ sufficiently small, we have 
$\int_\mathcal{O} |(y-\bar{y}- c)^+|^2= 0$.
Hence $\varphi(u - \bar{u}) \leqslant c$ in $\overline{\mathcal{O}}$ 
and similarly it follows that $\varphi(u - \bar{u}) \geqslant -c$ in 
$\overline{\mathcal{O}}$, $\| u - \bar{u} \|_{*} \leqslant \| v - \bar{v} \|_{*}$. 
We conclude that $\mathcal{A}_0$ is quasi-$m$-accretive in $C(\overline{\mathcal{O}})$ as claimed.

By Crandall-Liggett theorem (see \cite{Bar10}, p. 131) $\mathcal{A}_0$ generates a 
$C_0$-semigroup of $\alpha$-quasi-contractions $(\widetilde{V}_t)_{t \geqslant 0}$ on 
$C(\overline{\mathcal{O}})$ (endowed with the norm $\| \cdot \|_*$), 
given by the exponential formula
$$
\widetilde{V}_tf = \mathop{\lim}\limits_{n\to\infty}(I + \dfrac{t}{n}\mathcal{A}_0)^{-n}f \; {\rm in} \; C(\overline{\mathcal{O}}) \; 
{\rm for \; all} \; f \in C(\overline{\mathcal{O}}) \; \mbox{uniformly in}\; t\; \mbox{on compact intervals}.
$$
Since $(I + \dfrac{t}{n}\mathcal{A}_0)^{-1}f = (I + \dfrac{t}{n}\mathcal{A})^{-1}f$ for all $t > 0$ 
and $f \in C(\overline{\mathcal{O}})$, and 
$V_tf = \mathop{\lim}\limits_{n\to \infty}(I + \dfrac{t}{n}\mathcal{A})^{-n}f$ in $L^2(\mathcal{O})$ for $f \in L^2(\mathcal{O})$, 
we infer that
$$
V_tf = \widetilde{V}_tf \; \mbox{ for   all } \; f \in C(\overline{\mathcal{O}}) \; {\rm and} \; t \geqslant 0.
$$
Therefore for all $f, \bar{f} \in C(\overline{\mathcal{O}})$
$$
\| V_t f - V_t \bar{f} \|_{*}   
\leqslant e^{\alpha t} \| f - \bar{f} \|_{*}, \; t \geqslant 0, \; \mbox{\rm and } 
\; \mathop{\lim}\limits_{t \searrow 0}\| V_tf - f \|_{*} = 0
$$
as claimed. The fact that $V_t(S) \subseteq S$ is a direct consequence of assertion $(iii)$ in Proposition 2.1.
\end{proof}

\begin{remark}
The results from Section 2 are valid for more general functions $\beta$ which satisfies the following condition:
{\rm $\beta$ is continuous and $u\longmapsto \beta(u)+\gamma u$
is monotonically nondecreasing on $[0, \infty)$, where $\gamma$ is defined by $(1.3).$}
\end{remark}

\section{Negative definite properties of the nonlinear evolution equation} 

Let $M := M(\overline{\mathcal{O}})$ be the set of all finite positive measures on $\overline{\mathcal{O}}$, 
endowed with the weak topology. 
It is a locally compact topological space with a countable base.
In order to construct a semigroup of (linear) kernels on $M$, induced by $(V_t)_{t \geqslant 0}$, 
we start with preliminaries on the harmonic analysis of  the semigroup $S := C_+(\overline{\mathcal{O}})$; 
for more details see \cite{BeChRe84}, \cite{Fi88}, and \cite{Dy02}.

A real-valued function $\varphi : S \longrightarrow \mathbb{R}$ is called {\it negative definite} if for each $n \geqslant 2$, 
$\alpha_1, \ldots, \alpha_n \in \mathbb{R}$ with $\alpha_1 + \ldots + \alpha_n = 0$, and $u_1, \ldots, u_n \in S$ 
we have $\mathop{\sum}\limits_{i,j}\alpha_i \alpha_j \varphi(u_i + u_j) \leqslant 0$.

A function $\varphi : S \longrightarrow \mathbb{R}$ is named {\it positive definite} provided that for each $n \geqslant 1$, $\alpha_1, \ldots, \alpha_n \in \mathbb{R}$, 
and $u_1, \ldots, u_n \in S$ we have $\mathop{\sum}\limits_{i,j}\alpha_i \alpha_j \varphi(u_i + u_j) \geqslant 0$.

Basic properties of the negative and positive definite functions are the following:\\

\noindent
$(3.1)\:$  If $\varphi : S \longrightarrow \mathbb{R}$ is linear or constant then it is negative definite.
The linear combination with positive coefficients of negative (resp. positive) definite functions is also negative (resp. positive) definite.
The pointwise limit of a sequence of negative (resp. positive) definite functions is negative (resp. positive) definite.\\[-2mm]

\noindent
$(3.2)\:$  Let $\varphi : S \longrightarrow \mathbb{R}$. 
Then $\varphi$ is negative definite if and only if $e^{-t\varphi}$ is positive definite for all $t > 0$.\\

We present now two results which are  essentially used in the semigroup approach to the construction of the superprocesses
(see e.g., \cite{Fi88} and \cite{Be11}).

\begin{lem}  
{ If $a \in \mathbb{R}_+$, $b \in \mathbb{R}$ and $\varphi : S \to \mathbb{R}$ is negative definite, 
then the function $- \beta \circ \varphi + a \varphi + b$ is also negative definite.} 
\end{lem}

\begin{proof} 
By $(3.2)$ the function $e^{-s\varphi}$ is positive definite for all $s \geqslant 0$.
So, by $(3.1)$ the function $(1 - e^{-s\varphi})$ is negative definite and therefore $\int_{0}^{\infty}(1 - e^{-s\varphi})\eta (ds)$ has the same property. 
Again by  $(3.1)$ we conclude that $- \beta \circ \varphi + a\varphi + b$ is negative definite.
\end{proof}

It is easy to see that the exponential mapping $e_f : M \longrightarrow [0, 1]$ 
is a Borel measurable function on $M$ and that 
$S \ni f \longmapsto e_f(\mu)$ is positive definite on $S$ for each $\mu \in M$.
Further we denote by $\cb(M)$ the Borel $\sigma$-algebra on $M$.

\begin{prop}   
{\rm (\cite{Fi88}, \cite{Be11}).} 
{The following assertions hold.

$(i)$ Let $\xi$ be a finite, positive measure on $(M, \mathcal{B}(M))$ and consider the map 
$\varphi_{\xi} : S \longrightarrow \mathbb{R}_+$ defined as
$$
\varphi_{\xi}(f) = \int_{M}e_f(\mu)\xi (d\mu), \; f \in S.
$$

Then $\varphi_{\xi}$ is positive definite and if $(f_n)_n \subset S$, $\mathop{\lim}\limits_{n\to\infty}\| f_n \|_{C(\overline{\mathcal{O}})} = 0$,
then $\mathop{\lim}\limits_{n\to\infty}\varphi_{\xi}(f_n) = \varphi_{\xi}(0)$.

$(ii)$ Let $\varphi : S \longrightarrow [0, 1]$ be a positive definite map such that $\mathop{\lim}\limits_{n\to\infty}\varphi(\frac{1}{n}) = \varphi(0)$.

Then there exists a unique finite positive measure $\xi$ on $(M, \mathcal{B}(M))$ such that $\varphi = \varphi_{\xi}$, that is
$$
\varphi(f) = \int_M e_f \, d\xi \; \mbox{ for  all } f \in S.
$$}
\end{prop}

We can state now the main result of this section.

\begin{prop}   
{\it The following assertions hold.

$(i)$ For each $t \geqslant 0$ and $x \in \overline{\mathcal{O}}$ the function $S \ni f \longmapsto V_tf(x)$ is negative definite.

$(ii)$ For each $t \geqslant 0$ there exists a unique sub-Markovian kernel $Q_t$ on $(M, \mathcal{B}(M))$ such that
$$
Q_t(e_f) = e_{V_tf} \; for \; all \; f \in S. \leqno(3.3)
$$

$(iii)$ The family $(Q_t)_{t \geqslant 0}$ induces a Feller semigroup
on $C_0(M)$ 
(:= the space of real-valued continuous functions on $M$, vanishing at infinity).}
\end{prop}

\begin{proof}  
$(i)$  We consider the iteration process:
$$
\left\{ \begin{array}{ll}
\dfrac{\partial u_{n + 1}}{\partial t} - \frac{1}{2}  \Delta u_{n + 1} + \alpha u_{n + 1} = 0 & {\rm in} \; (0, T) \times \mathcal{O},\\[3mm]
\dfrac{\partial u_{n + 1}}{\partial \nu} + \beta(u_n) = g & {\rm on} \; (0, T) \times \Gamma,\\[3mm]
u_{n + 1}(0) = f & {\rm in} \; \mathcal{O}, \; f \in S.
\end{array}\right. \leqno(3.4)
$$

We claim that for $n \to \infty$, $u_n \longrightarrow u$ in 
$L^2(0, T; H^{1}(\mathcal{O}))$ and 
$\beta(u_n) \longrightarrow \beta(u)$ in $L^2((0, T) \times \Gamma)$, 
where $u$ is the solution of $(1.1)$.
Indeed, for each $v\in L^2(\Gamma)$ we denote by $F(v)$ the trace $\tau(u)$ of $u$, solution to 
$$
\left\{ \begin{array}{ll}
\dfrac{\partial u}{\partial t} - \frac{1}{2} \Delta u + \alpha u = 0 & {\rm in} \;  [0, T) \times \mathcal{O},\\[4mm]
\dfrac{\partial u}{\partial \nu} + \beta(v) = g & {\rm on} \; [0, T) \times \Gamma,\\[4mm]
u(0, \cdot) = f & {\rm in} \; \mathcal{O}.
\end{array}\right.
$$
Here we take the norm of  $H^{1}(\mathcal{O})$ as 
$|| u ||_{H^{1}(\mathcal{O})} := (\int_{\mathcal{O}} |\nabla u|^2 + 2\alpha u^2) dx)^{\frac{1}{2}}.$
We prove first that for $||\beta||_{\rm Lip}$ sufficiently small, $F$ is a contraction on 
$L^2((0,T)\times \Gamma)$.
If $v, \bar{v} \in L^2(\Gamma)$ and  $u, \bar{u}$  are related to $v$ and respectively $\bar{v}$ as before, then
$$
\left\{ \begin{array}{ll}
\dfrac{\partial}{\partial t}  (u- \bar{u}) - \frac{1}{2} \Delta (u- \bar{u})+ \alpha (u- \bar{u}) = 0 & {\rm in} \;  [0, T) \times \mathcal{O},\\[4mm]
\dfrac{\partial}{\partial \nu} (u- \bar{u})  + \beta(v) - \beta(\bar{v})= g & {\rm on} \; [0, T) \times \Gamma,\\[4mm]
(u- \bar{u}) (0, \cdot) = 0 & {\rm in} \; \mathcal{O}.
\end{array}\right.
$$
This yields 
$$
\frac12 | u(t) - \bar{u}(t)|^2_{L^2(\mathcal{O})} + 
\frac{1}{2}   \int_0^t   \int_{\mathcal{O}} |\nabla(u - \bar{u} )|^2 ds dx
+ \alpha  \int_0^t \int_{\mathcal{O}}  | u - \bar{u} |^2 ds dx \leqslant
\leqno{(3.5)}
$$
$$
 \int_0^t \int_{\Gamma} |\beta(v)-\beta(\bar{v})|   | u - \bar{u} | ds d\sigma \leqslant
 $$
 $$
 ||\beta||_{\rm Lip} \int_0^t  ||v- \bar{v}||_{L^2(\Gamma)}\,  ||u- \bar{u}||_{L^2(\Gamma)} ds
 \leqslant
\gamma^{-\frac{1}{2}}  ||\beta||_{\rm Lip}\,  ||v- \bar{v}||_{L^2((0, T) \times \Gamma)}  
||u- \bar{u}||_{L^2(0,T; H^1(\mathcal{O}))},
$$
where $\gamma$ is given by $(1.3)$.
Hence 
$||u-\bar{u}||_{L^2(0,T; H^1(\mathcal{O}))}\leqslant
2 \gamma^{-\frac{1}{2}}    ||\beta||_{\rm Lip}  ||v-\bar{v}||_{L^2((0, T) \times \Gamma)}  $
and this yields by the trace theorem
$||u-\bar{u}||_{L^2((0, T) \times \Gamma)} 
\leqslant
\frac 2 \gamma   ||\beta||_{\rm Lip}  ||v-\bar{v}||_{L^2((0, T) \times \Gamma)}  $.
Consequently  $F$ is a contraction for 
$$
 ||\beta||_{\rm Lip} <  \frac\gamma 2  \leqno{(3.6)}
$$
and so, the sequence $(u_n)_{n\geqslant 0}$ which is defined by $\tau(u_{n+1})= F(\tau(u_n))$, is strongly convergent in $L^2(\Gamma)$
to $\tau(u)$.
This implies also by $(3.5)$  that $(u_n)_{n\geqslant 0}$ is convergent in
$L^2(0, T; H^1(\mathcal{O}) )\cap C([0,T]; L^2(\mathcal{O})).$
Now we can get rid of condition $(3.6)$ by rescaling equation $(1.1)$ via the transformation $t\longrightarrow \lambda s$.\\

\noindent
{\bf Representation of $u_{n + 1}$ as the solution of an integral equation}. 
Define the linear continuous operator $A : H^1(\mathcal{O})  \longrightarrow (H^1(\mathcal{O}))'$ as
$$
{}_{(H^1(\mathcal{O}))'}\langle Au, \psi \rangle_{H^1(\mathcal{O})} :=
\frac{1}{2}  \int_{\mathcal{O}}(\nabla u \cdot \nabla \psi + 2 \alpha u) dx \; 
\mbox{ for  all } \; u, \psi \in H^1(\mathcal{O}),
$$
and also the mapping $\bar{\beta} : H^1(\mathcal{O}) \longrightarrow (H^1(\mathcal{O}))'$ as
$$
{}_{(H^1(\mathcal{O}))'}\langle \bar{\beta}(u), \psi \rangle_{H^1(\mathcal{O})} = 
- \frac{1}{2}   \int_{\Gamma} (\beta(u)-g) \tau(\psi) d\sigma \; \mbox{ for all } \; u, \psi \in H^1(\mathcal{O}),
$$
where $\tau(\psi) \in L^2(\Gamma)$ is the trace of $\psi$ on $\Gamma$.
Then $(3.4)$  is equivalent with
$$
\left\{ \begin{array}{ll}
\dfrac{\partial u_{n + 1}}{\partial t} = - A u_{n + 1} + \bar{\beta}(u_n) & {\rm in} \; (0, T) \times \mathcal{O},\\[3mm]
u_{n + 1}(0) = f & {\rm in} \; \mathcal{O},
\end{array}\right.
$$
which may be rewritten as (d'Alembert formula)
$$
\langle \psi, u_{n + 1}(t)\rangle_2  = \langle e^{At}\psi, f \rangle_2 
-  \frac{1}{2}  \int_{0}^{t} \int_{\Gamma} \tau(e^{A(t - s)}\psi) \, (\beta(u_n(s)) - g)\,  d\sigma ds 
\leqno{(3.7)}
$$
$$
\hspace*{93mm} {\rm for \; all} \; \psi \in H^1(\mathcal{O}), \; 0 < t < T, \; n \in \mathbb{N}. 
$$
By Lemma 3.1 and $(3.7)$  it follows inductively that the function 
$f \longmapsto \langle \psi, u_{n + 1}(t) \rangle_2$ is negative definite for all $\psi \in H^1(\mathcal{O})$, $\psi \geqslant 0$.
Passing to the limit and using (3.1) we get  that 
$u := \mathop{\lim}\limits_{n\to\infty} u_n$ has the same property and consequently assertion $(i)$ holds.

$(ii)$  From  $(i)$ and $(3.2)$  it follows that for each $\mu \in M$ the map
$$
f \longmapsto e_{V_tf}(\mu), \; f \in S,
$$
is positive definite and by Proposition 2.2 we have $\mathop{\lim}\limits_{n\to\infty} e_{V_t\frac{1}{n}}(\mu) = e_{V_t0}(\mu)$.
Proposition 3.2 $(ii)$ implies the existence of a finite measure $Q_{t, \mu}$ on $M$ such that
$$
Q_{t, \mu}(e_f) = e_{V_tf}(\mu) \; {\rm for \; all} \; f \in S.
$$
We have $Q_{t, \mu}(1) = Q_{t, \mu}(e_0) = e_{V_t0}(\mu) \leqslant 1$ since $V_t0 \geqslant 0$, so $Q_{t, \mu}$ 
is a sub-probability on $M$. The map $\mu \longmapsto Q_{t, \mu}(e_f)$ is $\mathcal{B}(M)$-measurable for all $f \in S$ 
and since the family $\{ e_f : f \in S \}$ generates the $\sigma$ - algebra $\mathcal{B}(M)$, 
we deduce by a monotone class argument that $\mu \longmapsto Q_{t, \mu}(F)$ is also $\mathcal{B}(M)$-measurable 
for all positive $\mathcal{B}(M)$-measurable function $F$ on $M$. Consequently
$$
F \longmapsto Q_tF(\mu) := Q_{t, \mu}(F), \; \mu \in M, 
$$
defines a sub-Markovian kernel on $M$.

$(iii)$ Since the family $(V_t)_{t \geqslant 0}$ is a (nonlinear) semigroup it follows that $(Q_t)_{t \geqslant 0}$ is semigroup of sub-Markovian kernels on $M$.
Let $S' := \{ f \in S : f > 0 \}$, $\; \widehat{S'} := \{ e_f : f \in S' \}$ and $[\widehat{S'}]$ the linear space spanned by $\widehat{S'}$. 
Since by Proposition 2.1 $(iii)$  
$V_t(S') \subseteq S'$ for all $t \geqslant 0$, it follows that $Q_t(\widehat{S'}) \subseteq \widehat{S'}$ 
and therefore $Q_t([\widehat{S'}]) \subset [\widehat{S'}]$. 
Because $[\widehat{S'}]$ is an algebra, 
separating the points of $M$ (by Stone-Weierstrass Theorem), 
it is a dense subset of $C_0(M)$ and therefore $Q_t(C_0(M)) \subseteq C_0(M)$ for all $t \geqslant 0$.
By Proposition 2.2 we have $\mathop{\lim}\limits_{t \searrow 0}\| V_tf - f \|_{C(\overline{\mathcal{O}})}=0$ 
and thus $\mathop{\lim}\limits_{t \searrow 0} Q_t(e_f)(\mu) = e_f(\mu)$ for all $\mu \in M$, hence
$Q_t(F)$ converges to $F$  for every $F \in [\widehat{S'}]$,  pointwise on $M$, as $t$ decreases to zero. 
We conclude that the above pointwise convergence holds for every $F \in C_0(M)$ 
and therefore $(Q_t)_{t \geqslant 0}$ is a $C_0$-semigroup on $C_0(M)$.
\end{proof}

\section{Measure-valued branching processes}

\subsection{Branching processes on the closure of the open set}  

In this subsection we associate to the equation $(1.1)$   a branching process  with state space the set $M = M(\overline{\mathcal{O}})$ of all finite measures on the closure of the open set $\mathcal{O}$.

Recall that if $p_1, p_2$ are two finite measures on $M $, 
then their convolution $p_1 \ast p_2$ is the finite measure on $M$ defined for every $h \in \mathcal{B}_+(M)$ 
(:=   the set of all positive $\mathcal{B}(M)$-measurable functions on $M$)
by
$$
\int_M p_1 \ast p_2 (d\nu) h(\nu) := \int_M p_1(d\nu_1)\int_M p_2(d\nu_2)h(\nu_1 + \nu_2).
$$

A bounded kernel $Q$ on $(M, \mathcal{B}(M))$ is called {\it branching kernel} provided that
$$
Q_{\mu + \nu} = Q_{\mu} \ast Q_{\nu} \; {\rm for \; all} \; \mu, \nu \in M,
$$
where $Q_{\mu}$ denotes the measure on $M$ such that $\int h d Q_{\mu} = Q h(\mu)$ for all $h \in \mathcal{B}_+(M)$.\\

A right (Markov) process with state space $M$ is called {\it branching process} provided that its transition function is formed by branching kernels. 
The probabilistic interpretation of this analytic branching property of a process is as follows:
{\it if we take two independent versions $X$ and $X'$ of the process, starting respectively from two measures $\mu$ and $\mu'$, 
then $X + X'$ and the process starting from $\mu + \mu'$ are equal in distribution.}

\begin{exam} 
Let $(P_t)_{t \geqslant 0}$ be the transition function of the reflected Brownian motion on $\overline{\mathcal{O}}$.

(i) If $t \geqslant 0$ and $\alpha \geqslant 0$ then the kernel $Q_t^0$ on $M$ defined as
$$
Q_t^0 F(\mu) := F(\mu \circ e^{-\alpha t}P_t), \; F \in \mathcal{B}_+(M), \; \mu \in M,
$$
is a branching kernel.
(ii) The linear semiflow on $M$, $\mu \longmapsto \mu \circ e^{-\alpha t}P_t$,  
is  a continuous path branching (deterministic) process with transition function  $(Q_t^0)_{t \geqslant 0}$.
If $f \in C_+^2(M)$, $F := e_f$,  and $\mu \in M$, then there exists 
$$
\mathop{\lim}\limits_{t \searrow 0} \dfrac{Q_t^0 F(\mu) - F(\mu)}{t} =: L^0 F(\mu)
$$ 

\noindent and we have 
$L^0 F(\mu) = \ds\int_{\overline{\mathcal{O}}} \mu (dx)[\frac{1}{2} \Delta F'(\mu, x) - \alpha F'(\mu, x)]$, 
where recall that the {\rm variational derivative} of a function $F : M \longrightarrow \mathbb{R}$ is
$$
F'(\mu, x) := \mathop{\lim}\limits_{t \searrow 0}\dfrac{1}{t}(F(\mu + t\delta_x) - F(\mu)), \; \mu \in M, \; x \in \overline{\mathcal{O}}.
$$
\end{exam}

Now, we present the main result of this paper. Let $(Q_t)_{t \geqslant 0}$ be the 
Feller semigroup 
given by assertion $(iii)$ of Proposition 3.3, induced by the solution of $(1.1)$.

\begin{thm}  
There exists a branching Markov process $X = (X_t)_{t \geqslant 0}$ with state space $M$, such that the following assertions hold.

$(i)$ $X$ is a Hunt process with
transition function $(Q_t)_{t \geqslant 0}$, 
or equivalently, for every $f \in S$, the solution $(V_tf)_{t \geqslant 0}$ of the nonlinear parabolic problem $(1.1)$ has the representation
$$
V_tf(x) = - \ln \mathbb{E}^{\delta_x}(e_f(X_t); t < \zeta), \; t \geqslant 0, \; x \in \overline{\mathcal{O}}, \leqno{(4.1)}
$$

\noindent where $\zeta$ denotes the life time of $X$.

$(ii)$ Let $(L, D(L))$ be the infinitesimal generator of $X$, that is $Q_t = e^{tL}$, $t \geqslant 0$, i.e. $(L, D(L))$ 
is the generator of the $C_0$-semigroup $(Q_t)_{t \geqslant 0}$ on $C_0(M)$. If $f \in D(\mathcal{A})$, $F:= e_f$, 
and $\mu \in H^1(\mathcal{O})$ with $\Delta \mu \in H$, 
then there exists $\mathop{\lim}\limits_{t \searrow 0}\dfrac{Q_t F(\mu) - F(\mu)}{t} =: L F(\mu)$  and
$$
\begin{array}{cc}
\ds L F(\mu) = \langle \frac{1}{2}  \Delta \mu - \alpha \mu, F'(\mu) \rangle_2 - 
 \frac{1}{2}   \int_{\Gamma} \dfrac{\partial \mu}{\partial \nu}  F'(\mu)\,  d\sigma  \\[3mm]
+\ds\int_{\Gamma} \mu |_{\Gamma} (dy)
\left(\int_{0}^{\infty}[F(\mu + s\delta_y) - F(\mu)] \eta (ds) -  \frac{1}{2}  g(y) F(\mu)- bF'(\mu, y)\right).
\end{array} \leqno(4.2)
$$
\end{thm}

\begin{proof}  
$(i)$   By assertion $(ii)$ of Proposition 3.3 we have $Q_t(e_f) = e_{V_tf}$ for all $t \geqslant 0$ and $f \in S$. 
Using also a monotone class argument it follows that $Q_t$ is a branching kernel on $M$ for all $t \geqslant 0$; 
for details see \cite{Wat68}, \cite{Li11}, \cite{Fi88}, and \cite{Be11}. 
From the basic existence theorem for Hunt  processes (cf. Theorem 9.4 in \cite{BlGe68}), 
since $(Q_t)_{t \geqslant 0}$ is a Feller semigroup
according to Proposition 3.3 $(iii)$, 
there exists the claimed Hunt  process with transition function  $(Q_t)_{t \geqslant 0}$. 
Hence $Q_t F(\mu) = \mathbb{E}^{\mu}(F(X_t); t < \zeta)$ for all $\mu \in M$ and $F \in \mathcal{B}_+(M)$. 
The equality $(4.1)$ follows now by $(3.3)$.

$(ii)$  Since $f$ belongs to the domain of $\mathcal{A}$, given by $(2.1)$, using $(2.3)$ we get
$$
L F(\mu) = 
\dfrac{d(e_{V_tf}(\mu))}{dt} = - e_f(\mu) \left\langle \mu, \dfrac{d^+V_tf}{dt}(0) \right\rangle_2 
$$
$$
= F(\mu)\langle \mu, \mathcal{A}f \rangle_2 =
 - F(\mu)\left[\int_{\mathcal{O}} \frac{1}{2} \mu \Delta f - \alpha \int_{\mathcal{O}} \mu f\right].
$$
By the Green formula and again by $(2.1)$
$$
\int_{\mathcal{O}} \mu \Delta f = \int_{\mathcal{O}} f \Delta \mu + \int_{\Gamma} \dfrac{\partial f}{\partial \nu}\mu\, d\sigma  
- \int_{\Gamma}f\dfrac{\partial \mu}{\partial \nu} \, d\sigma
= \int_{\mathcal{O}} f \Delta \mu + \int_{\Gamma} [(g - \beta(f))\mu - f \dfrac{\partial \mu}{\partial \nu}]\, d\sigma.
$$
We have also
$$
F(\mu) \int_{\Gamma} \beta(f)\mu\,  d\sigma = 
F(\mu)\int_{\Gamma}\mu(y)\left[ \int_{0}^{\infty}(e^{-sf(y)} - 1)\eta (ds) - bf(y) \right]\sigma (dy)
$$
$$
=\int_{0}^{\infty} \eta (ds)\int_{\Gamma} \mu(y) F(\mu) \left[ F(s\delta_y) - 1 \right]\sigma (dy) - b F(\mu) \int_{\Gamma}\mu f \, d\sigma 
$$
$$
= \int_{0}^{\infty} \eta (ds) \int_{\Gamma} \mu(y) \left[ F(\mu + s\delta_y) - F(\mu) \right]\sigma (dy) - b F(\mu) \int_{\Gamma} \mu f \, d\sigma
$$
$$
= \int_{\Gamma} \mu|_{\Gamma} (dy) \left( \int_{0}^{\infty} [F(\mu + s\delta_y) - F(\mu)] \eta (ds) - b F(\mu) f(y) \right).
$$
Since $F'(\mu, \cdot) = - f F(\mu)$  we conclude that  $(4.2)$  holds.
\end{proof}

\subsection{Branching processes on the boundary} 

Define the nonlinear operator $\Lambda : D(\Lambda) \subseteq L^2(\Gamma) \longrightarrow L^2(\Gamma)$ as
$$
\left\{ \begin{array}{l}
D(\Lambda) := \{ \varphi \in H^{\frac{1}{2}}(\Gamma) : \mbox{ there exists } 
u \in H^1(\mathcal{O}) \mbox{ s.t. } \frac{1}{2}  \Delta u - \alpha u = 0 \; {\rm in} \; \mathcal{O},\\[3mm] 
\hspace*{\fill} u|_{\Gamma} = \varphi\  , \; \dfrac{\partial u}{\partial \nu} \in L^2(\Gamma)\},\\[3mm]
\Lambda \varphi\ := - \dfrac{\partial u}{\partial \nu} - \beta(\varphi) \; {\rm for \; all} \; \varphi \in D(\Lambda),\\[3mm]
 {\rm where} \; u \in H^1(\mathcal{O})  \mbox{ is s.t. }  \frac{1}{2} \Delta u - \alpha u = 0  \mbox{ in } \mathcal{O}
\mbox { and } \; u|_{\Gamma} = \varphi.
\end{array}\right. \leqno(4.3)
$$
The exact meaning of $(4.3)$ is the following.
$D(\Lambda)$ is the space of all $\varphi \in H^{\frac{1}{2}}(\Gamma)$ 
with the property that there are $u \in H^1(\mathcal{O})$ and  $\eta \in L^2 (\Gamma)$ such that
$$
\int_{\mathcal{O}} u(\Delta \psi - 2 \alpha \psi)dx = 
\int_{\Gamma} \varphi\frac{\partial\psi}{\partial\nu} d \sigma  - 
\int_\Gamma \eta \psi d\sigma
\; \mbox{ for all } \; \psi \in H^2(\mathcal{O}). \leqno(4.4)
$$
(In fact $\eta = \dfrac{\partial u}{\partial \nu}$ and $\varphi = \tau(u)$.)
The operator $\Lambda : D(\Lambda)  \longrightarrow  L^2(\Gamma)$ is defined by 
$$
\int_\Gamma \Lambda \varphi \, \xi\, d \sigma = - \int_\Gamma (\eta + \beta(\varphi)) \xi d \sigma  \; 
\mbox{ for all }\;  \xi \in L^2(\Gamma),  \varphi \in D(\Lambda), \leqno (4.5)
$$
where $\eta = \dfrac{\partial u}{\partial \nu}$ is defined by $(4.4)$.


\vspace{3mm}

\begin{lem} 
The operator $\Lambda$ is maximal monotone in $L^2(\Gamma)$.
\end{lem}

\begin{proof}
 It suffices to show that $R(I + \Lambda) = L^2(\Gamma)$ and 
 $\| (I + \Lambda)^{-1} \|_{{\rm Lip}(L^2(\Gamma))} \leqslant 1$. 
 On the other hand, 
  for each  $f \in L^2(\Gamma)$ 
 the equation $\varphi + \Lambda \varphi = f$ reduces to
$$
\left\{ \begin{array}{ll}
\frac{1}{2} \Delta u - \alpha u = 0 & {\rm in} \; \mathcal{O},\\[3mm]
u + \dfrac{\partial u}{\partial \nu} + \beta(u) = f & {\rm on} \; \Gamma,\\[3mm]
u|_{\Gamma} = \varphi.
\end{array}\right. \leqno(4.6)
$$
Of course $(4.6)$ should be taken in the sense of $(4.4)$-$(4.5)$, that is, for all $\psi \in H^2(\mathcal{O})$
$$
\int_{\mathcal{O}} u(\Delta \psi - 2 \alpha \psi) dx = 
\int_\Gamma \varphi \frac{\partial\psi}{\partial\nu} d \sigma - \int_\Gamma(f - \beta (\varphi) - \varphi) \psi d \sigma. \leqno(4.7)
$$

Equation $(4.6)$ 
(without the Dirichlet boundary condition $u|_{\Gamma} = \varphi$)
has a unique weak solution $u \in H^1(\mathcal{O})$, that is
$$
\int_{\mathcal{O}}(\nabla u \cdot \nabla \psi + 2\alpha u\psi)dx + \int_{\Gamma} (u + \beta(u) - f)\psi d\sigma = 0 
\quad {\rm for \; all} \; \psi \in H^1(\mathcal{O}). \leqno(4.6')
$$
Here is the argument. 
We can rewrite $(4.6')$  as
$$
L_1 v +L_2  v=f,
$$
where $L_1 v:= \frac{\partial v}{\partial \nu} - \gamma v$
and
$L_2 v:= \beta(v) + \gamma v$,
with 
$$
D(L_1):= \{ v\in L^2(\Gamma): v=u|_\Gamma, \ \frac 1 2 \Delta u -\alpha v =0 \mbox{ in } \mathcal{O},\
\frac{\partial v}{\partial \nu}\in  L^2(\Gamma) \}.
$$
By the Lax-Milgram Lemma it follows that $R(I+L_1)=   L^2(\Gamma)$ and so, 
$L_1$ is $m$-accretive (or equivalently, it is maximal monotone).
It is  clear by the monotonicity of the function $u\longmapsto \beta(u)+\gamma u$ 
that  also $R(I+L_2)=   L^2(\Gamma)$.
Then by Rockafellar's perturbation result (see \cite{Bar10}, page 44), 
it follows that $\Lambda= L_1+L_2$ is maximal monotone ($m$-accretive) and so
$f\in R(I+\Lambda)$, as claimed.

This clearly implies via Green's formula that $(u, \varphi=u|_{\Gamma})$ satisfy $(4.7)$.
Hence $R(I + \Lambda) = L^2(\Gamma)$. 
By $(4.4)$ we have for all 
$f, \bar{f} \in L^2(\Gamma)$ and $u, \bar{u}$ the corresponding solutions of (3.2)
$$
\left\{ \begin{array}{ll}
\frac{1}{2} \Delta(u - \bar{u}) - \alpha(u - \bar{u}) = 0 & {\rm in} \; \mathcal{O},\\[3mm]
u - \bar{u} + \dfrac{\partial}{\partial \nu}(u - \bar{u}) + \beta(u) -\beta(\bar{u}) = f - \bar{f} & {\rm on} \; \Gamma.
\end{array}\right. \leqno(4.8)
$$
This yields, again via the Green's formula,
$$
\int_{\mathcal{O}}(|\nabla(u - \bar{u})|^2 + 2\alpha(u - \bar{u})^2)dx + \int_{\Gamma}[(u - \bar{u})^2 +
 \beta(u) - \beta(\bar{u})(u - \bar{u})]d\sigma = \int_{\Gamma}(f - \bar{f})(u - \bar{u})d\sigma.
$$
As in the proof of the monotonicity of the operator $\mathcal{A}$ in Section 2, 
using the monotonicity of the map $r \longmapsto \beta(r) + \gamma  r$ and condition $(1.3)$, we get
$$
\| f - \bar{f} \|_{L^2(\Gamma)} \cdot \| u - \bar{u} \|_{L^2(\Gamma)} 
$$
$$
\geqslant  \| u - \bar{u} \|^2_{L^2(\Gamma)} + \| |\nabla(u - \bar{u})| \|^2_{L^2(\mathcal{O})} + \| u - \bar{u} \|^2_{L^2(\mathcal{O})} - 
\gamma \| u - \bar{u} \|^2_{L^2(\Gamma)} \geqslant \| u - \bar{u} \|^2_{L^2(\Gamma)}
$$
and therefore
$$
\| u - \bar{u} \|_{L^2(\Gamma)} \leqslant \| f - \bar{f} \|_{L^2(\Gamma)} \quad {\rm for \; all} \; f, \bar{f} \in L^2(\Gamma),
$$
as claimed.
\end{proof}

By the generation  theory of $C_0$-semigroups of contractions (as in Section 2, see, e.g., \cite{Bar10}), 
we infer that the equation
$$
\left\{ \begin{array}{ll}
\dfrac{d v(t)}{dt} + \Lambda v(t) = 0 & {\rm for \; all} \; t \geqslant 0,\\[3mm]
v(0) = f &
\end{array}\right. \leqno(4.9)
$$
has for each $f \in L^2(\Gamma)$ a unique solution $v(t) =: W_t f  = e^{- \Lambda t}f \in C([0, T]; L^2(\Gamma))$ for all $T > 0$,
given by the exponential formula
$$
W_t f  = \mathop{\lim}\limits_{n\to \infty}(I + \dfrac{t}{n}\Lambda)^{-n}f \quad {\rm for \; all\;}  t\geqslant 0, {\rm uniformly \; on 
\; compacts \; of} \; \mathbb{R}_+. \leqno(4.10)
$$
We have also for  $f \in D(\Lambda)$
$$
\dfrac{d^+ W_t f }{dt} + \Lambda W_t f  = 0 \, \mbox{ for  all } \; t > 0.
$$

\noindent
{\bf Remark} The above result remains true for any continuous and sub-linear function $\beta$ such that $r\mapsto \beta(r) +\gamma r$ is monotonically increasing.\\

Let us define the linear operator 
$\Lambda_1 : D(\Lambda_1) \subseteq L^2(\Gamma) \longrightarrow L^2(\Gamma)$ as
$$
\left\{ \begin{array}{l}
D(\Lambda_1) = \{ \varphi \in H^{\frac{1}{2}}(\Gamma), \; \frac{1}{2} \Delta u - \alpha u = 0, 
\; u|_{\Gamma} = \varphi, \; \dfrac{\partial u}{\partial \nu} \in L^2(\Gamma) \},\\[3mm]
\Lambda_1 \varphi := - \dfrac{\partial u}{\partial \nu}.
\end{array}\right.
$$
Since $\beta$ is Lipschitzian, we may define also the  Lipschitz operator 
$\Lambda_2 : L^2(\Gamma) \to L^2(\Gamma)$, $\Lambda_2(\varphi) := - \beta(\varphi)$. We have $D(\Lambda_1) = D(\Lambda)$, $\Lambda = \Lambda_1 + \Lambda_2$ and so, by 
$(4.9)$ we have
$$
W_t f = e^{- t \Lambda_1 t}f - \int_{0}^{t} e^{- (t-s) \Lambda_1}\Lambda_2 W_s f ds \quad {\rm in}\; \Gamma \; {\rm for \; all} \; t \geqslant 0, \leqno(4.11)
$$
where $ e^{- t \Lambda_1}$ is the $C_0$-semigroup of quasi-contractions generated by $-\Lambda_1$ on $L^2(\Gamma)$.

\begin{thm} 
There exists a branching Markov process $Z^\Gamma=(Z^\Gamma_t)_{t\geqslant 0}$ with state space the set
$M(\Gamma)$ (:= the set of positive finite measures on $\Gamma)$ such that
$$
W_t f= - \ln \Ee^{\delta_{\cdot}}( e_f (Z^{\Gamma}_t)) \;\; \mbox{ for all } t\geqslant 0 \mbox{ and } f\in C(\Gamma), \leqno{(4.12)}
$$
and the following assertions hold.

$(i)$  The process  $Z^\Gamma=(Z^\Gamma_t)_{t\geqslant 0}$ is precisely the $(Z, \alpha-\beta)$-superprocess on $M(\Gamma)$,
where  $Z=(Z_t)_{t\geqslant 0}$ is the process on the boundary, induced by the reflected Brownian motion.

$(ii)$ Let 
$(L^\Gamma, D(L^\Gamma))$ be the generator  on $C_0(M(\Gamma))$ of $Z^\Gamma=(Z^\Gamma_t)_{t\geqslant 0}$.
 If $f\in C^1(\Gamma)$ then $F:=e_f$ belongs to $D(L^\Gamma)$ and for every $\mu\in M(\Gamma)$ we have
$$
L^\Gamma F(\mu)= 
\int_{\Gamma}\mu (dy) \left( F(\mu) \frac{\partial u}{\partial \nu}(y) +
\int_{0}^{\infty}[F(\mu + s\delta_y) - F(\mu)]\eta(ds) - b F'(\mu, y)\right),\leqno(4.13)
$$
where $u\in H^1(\mathcal{O})$ is such that $\frac{1}{2} \Delta u-\alpha u=0$ and $u|_\Gamma=f.$
In particular, if $\mu \in H^1(\mathcal{O})$ with $\Delta \mu \in H$, then  we have
$$
L^\Gamma F(\mu|_{\Gamma}) = 
\langle \frac{1}{2} \Delta \mu - \alpha \mu, \overline{F}'(\mu|_{\Gamma}) \rangle_2 - \int_{\Gamma} \dfrac{\partial \mu}{\partial \nu} {F}'(\mu|_{\Gamma}) d\sigma \leqno(4.14)
$$
$$
+ \int_{\Gamma}\mu|_{\Gamma}(dy)\left( \int_{0}^{\infty}[F(\mu |_{\Gamma} + s\delta_y) - F(\mu |_{\Gamma})]\eta(ds) - b F'(\mu|_{\Gamma}, y) \right),
$$
where $\overline{F}$ is the extension of $F$ from $M(\Gamma)$ to $M(\overline{\mathcal{O}})$,  
defined as $\overline{F}:= e_u$.
\end{thm}

\begin{proof}
$(i)$  Let $(S_t)_{t\geqslant 0}$ be the transition function of the  the process on the boundary $Z=(Z_t)_{t\geqslant 0}$,
 induced by the reflected Brownian motion.
$(S_t)_{t\geqslant 0}$ induces a $C_0$-semigroup on $C(\Gamma)$ (see, e.g., \cite{SaUe65}) 
and its generator is 
$(\Lambda_1+\alpha, D(\Lambda_1))$:
$e^{-\Lambda_1 t}= e^{-\alpha t} S_t$, $t\geqslant 0$; see e.g., \cite{BeVl14}.
By the classical result on the existence of the continuous state branching processes 
(see Theorem 2.1 and Theorem 2.3 from \cite{Wat68}; 
see also \cite{Fi88},  \cite{Li11}, \cite{Be11}, and \cite{BeLuOp12}) there exists a branching process 
 $Z^\Gamma=(Z^\Gamma_t)_{t\geqslant 0}$ on $M(\Gamma)$, such that its transition function
 $(Q^\Gamma_t)_{t \geqslant 0}$ is a $C_0$-semigroup on $C_0(M(\Gamma))$, and we have
 $$
 Q^\Gamma_t(e_f)= e_{W^o_t f}, \; f\in C_+(\Gamma), t\geqslant 0, \leqno{(4.15)}
 $$
 where $(W^o_t f)_{t\geqslant 0}$ is the solution of the integral equation
 $$
 W^o_t f = S_t f + \int_0^t  S_{t-s} (\alpha-\beta)(W^o_sf)\, ds, \; t\geqslant 0.
 $$
The process process   $Z^\Gamma$ is the $(Z, \alpha-\beta)$-superprocess on $M(\Gamma)$.
By Proposition 3.1 from \cite{Be11} it follows that $(W^o_t f)_{t\geqslant 0}$ is a solution
of $(4.11)$ too, hence $W_t f= W^o_t f$ for all $t\geqslant 0$. 
Since $Q^\Gamma_t(e_f)= \Ee^{\cdot}(e_f(Z^\Gamma_t))$, the equality $(4.12)$ is a consequence of  $(4.15)$.

$(ii)$ The fact that $F=e_f\in D(L^\Gamma)$ follows from Theorem 2.3 in \cite{Wat68} since $f\in D(\Lambda_1)$ 
if $f\in C^1(\Gamma)$.
Arguing as in the proof of $(4.2)$ we have
$$
L^\Gamma F(\mu) = 
\dfrac{d(e_{W^o_tf}(\mu))}{dt} = - e_f(\mu) \mu (\dfrac{d^+ W^o_tf}{dt}(0)) =
F(\mu) \int_\Gamma \mu(dy) [\Lambda_1 f(y)+ \beta(f(y))]
$$
and we deduce further that $(4.13)$ and $(4.14)$ hold.
\end{proof}

\noindent
{\bf Final Remark.} 
$(i)$ Comparing the equalities $(4.2)$  and $(4.14)$, one can see that
$$
L\overline{F}(\mu|_\Gamma)= L^\Gamma {F} (\mu|_\Gamma), \leqno{(4.16)}
$$ 
with the notations and under the conditions of assertion $(ii)$ of Proposition 4.4, where recall that
$\overline{F}|_\Gamma=F$.
Observe also that if the measure $\mu$ has compact support in $\mathcal{O}$, then
$$
L{F}(\mu)= L^0 {F} (\mu), \leqno{(4.17)}
$$
under the conditions of assertion $(ii)$ of Theorem 4.2, where $L^0$ is defined in Example 4.1.

$(ii)$ The above assertion $(i)$ justifies  the probabilistic interpretation presented in the Introduction.
By $(4.17)$  the measure-valued branching process $X$ on $M(\overline{\mathcal{O}})$
behaves on  measures carried by $\mathcal{O}$  like the linear continuous semiflow, 
induced by the reflected Brownian motion,  presented in Example 4.1.
The non-local part of the generator $L$ of the branching process $X$,
as it is described  in  $(4.2)$ 
(the "L\'evy measure" part, see e.g. \cite{Sha88}),  
indicates that  the jumps of $X$ occur only on the set of measures, 
having non-zero traces on the boundary $\Gamma$. 
In particular, by $(4.16)$  the process $X$ behaves 
as the $(Z, \alpha- \beta)$-superprocess on measures carried by $\Gamma$.

$(iii)$ Recall that for  the Neumann problem, the boundary process $Z$  on $\Gamma$ 
is given by the time moments when the Brownian motion on $\mathcal{O}$  is reflected at the boundary.
Analogously, the  $(Z, \alpha-\beta)$-superprocess on $M(\Gamma)$ describes
the branching moments of the measure-valued process $X$, 
associated to the problem $(1.1)$.

$(iv)$  By  $(4.14)$  the infinitesimal operator of  the
$(Z, \alpha- \beta)$-superprocess, the measure-valued "boundary process",  
has no second order  (differential) term, 
as it happens in the classical case of the infinitesimal operator of the boundary process $Z$,
associated to the (linear) Neumann problem; cf. \cite{SaUe65}, p. 570.

$(v)$ Probabilistic representations for solutions of nonlinear Neumann boundary value problems
are related to the "catalytic  super-Brownian motion". 
Such a representation formula, similar to $(1.4)$,  is given in \cite{DeVo05}  for the nonnegative
solution of a mixed Dirichlet nonlinear Neumann boundary value problem; 
we thank P. J. Fitzsimmons  for suggesting us this connection. 

$(vi)$ Recall that the nonlinear part $-\beta$ in the Neumann boundary condition of the parabolic problem $(1.1)$ 
is classically  the branching mechanism of a superprocesses  associated to  the Dirichlet boundary value problem.
It is possible to construct branching processes starting with other types  of branching mechanisms 
(for relevant examples see \cite{BeLu16} and also \cite{BeDeLu15}) and note that the state space of these "non-local branching" processes is the set of all finite configurations of $\mathcal{O}$ 
(:= the set of all finite sums of Dirac measures concentrated in points of $\mathcal{O}$). 
It is known to solve the corresponding Dirichlet boundary value problem; see e.g. \cite{BeOp11} and \cite{BeOp14}.
However, it is a challenge to investigate the appropriate nonlinear parabolic problem, similar to $(1.1)$.\\

\noindent
{\bf Note added in proof.} 
This work is a version of the article with the same title [{\it J. Math. Anal. Appl.} {\bf 441} (2016), 167-182],
which details the proofs of a few technical results mentioned without proof in that work. 
Also, some inaccuracies were eliminated.

\vspace{5mm}



\begin{thebibliography}{10000000000}

\bibitem[Bar 10]{Bar10}
Barbu, V.,
{\it Nonlinear Differential Equations of Monotone Type in Banach Spaces}, Springer, 2010. 

\bibitem[BeChRe 84]{BeChRe84}
Berg, C., Christensen, J. P. R.,  and Ressel, P.,
{\it Harmonic Analysis on Semigroups}, Springer, 1984.

\bibitem[Be 11]{Be11} 
Beznea, L.,  
Potential theoretical methods in the construction of measure-valued branching processes, 
{\it J. European Math. Soc.} 
{\bf 13}  (2011), 685--707.



\bibitem[BeDeLu 15]{BeDeLu15}
Beznea, L.,  Deaconu, M., and Lupa\c scu, O.,
Branching processes for the fragmentation equation.
{\it Stochastic Processes and their Applications}  
{\bf 125} (2015),  1861--1885.

\bibitem[BeLu 16]{BeLu16}
Beznea, L. and  Lupa\c scu, O., 
Measure-valued discrete branching Markov processes.
{\it Trans. Amer. Math. Soc.}   {\bf 368}  (2016),  5153--5176.


\bibitem[BeLuOp 12]{BeLuOp12}
Beznea, L.,  Lupa\c scu,  O., and  Oprina, A.-G., 
A unifying construction for measure-valued continuous and discrete branching processes.
In  {\em  Complex Analysis and Potential Theory, CRM Proceedings and
Lecture Notes}, vol. {\bf 55}, Amer. Math. Soc., Providence, RI, 2012,  47--59.

\bibitem[BeOp 11]{BeOp11}
{Beznea L., and Oprina, A.-G.}, 
Nonlinear PDEs and measure-valued branching type processes.
{\it J. Math. Anal. Appl.} {\bf 384} (2011), {16--32}.


\bibitem[BeOp 14]{BeOp14}
{Beznea L., and Oprina, A.-G.}, 
 {Bounded and ${L}^p$-weak solutions for nonlinear equations of measure-valued branching processes}.
{\it Nonlinear Analysis} {\bf 107} (2014), {34--46}.

\bibitem[BeVl 14]{BeVl14}
Beznea, L. and Vl\u adoiu, S.,
{Markov processes on the Lipschitz boundary for the Neumann and Robin problems},  
{\it J. Math. Anal. Appl.} {\bf 455} (2017), 292--311.




\bibitem[BlGe 68]{BlGe68} Blumenthal, R. M. and  Getoor, R. K., 
{\it Markov Processes and Potential Theory} Academic Press, New York, 1968.

\bibitem[Brez 72]{Brez72}
Br\'ezis, H.,
Problemes unilateraux.
{\it J. Math. pures et appl.}
{\bf 51} (1972), 1--168.



\bibitem[DeVo 05]{DeVo05}
Delmas, J.-F. and Vogt, P., 
Non-linear Neumann's condition for the heat equation:
a probabilistic representation using catalytic super-Brownian
motion.
{\it Ann. I. H. Poincar\'e Ð PR} {\bf 41} (2005),  817--849


\bibitem[Dyn 02]{Dy02} 
Dynkin, E.  B., 
{\it Diffusions, Superdiffusions and Partial Differential Equations,}  
Amer. Math. Soc. Colloq. Publ. 50, Amer. Math. Soc., 2002.

\bibitem[Fitz 88]{Fi88}
Fitzsimmons, P. J., 
Construction and regularity of measure-valued Markov branching processes.
{\it Israel J. Math.} {\bf 64} (1988), 337--361.




\bibitem[LeGa 99]{LeGa99}
Le Gall, J.-F.,
{\it Spatial Branching Processes, Random Snakes and Partial Differential Equations} 
(Lectures in Mathematics ETH Z\"urich) Birkh\"auser, 1999.

\bibitem[Li 11]{Li11}
Li, Z. H., 
{\it Measure-Valued Branching Markov Processes,}  Probab. Appl., Springer, 2011.


\bibitem[SaUe 65]{SaUe65}  
Sato, K. and Ueno, T.,
Multi-dimensional diffusion and the Markov process on the boundary. 
{\it J. Math. Kyoto Univ.} {\bf 4} (1965),   529--605.

\bibitem[Sha 88]{Sha88} Sharpe, M.,  {\it General Theory of Markov Processes}, Academic Press, Boston, 1988.



\bibitem[Wat 68]{Wat68}
Watanabe, S., 
A limit theorem of branching processes and continuous state branching processes,
{\it J. Math. Kyoto Univ.} {\bf 8} (1968), 141--167.
\end{thebibliography}
\end{document}